\documentclass[preprint,12pt,authoryear]{elsarticle}

\usepackage{hyperref,natbib}
\usepackage{amssymb}
\usepackage{tikz,standalone}
\usepackage{amsmath}
\usepackage{amsthm}
\usepackage{appendix}
\usepackage{notation}
\usepackage{todonotes}

\newtheorem{theorem}{Theorem}[section]
\newtheorem{lemma}[theorem]{Lemma}
\newtheorem{corollary}[theorem]{Corollary}

\theoremstyle{definition}
\newtheorem{example}[theorem]{Example}

\DeclareMathOperator*{\ex}{\mathbb E}
\DeclareMathOperator*{\pr}{\mathbb P}
\DeclareMathOperator{\erf}{erf}

\DeclareRobustCommand{\VANDER}[3]{#2}

\journal{SPL: E-values Special Issue}

\begin{document}

\begin{frontmatter}

\title{A Generalisation of Ville's Inequality to Monotonic Lower Bounds and Thresholds}

\author[1,2]{Wouter M. Koolen} %
\author[3]{Muriel F. Pérez-Ortiz} %
\author[1,4]{Tyron Lardy} %

\affiliation[1]{organization={Centrum Wiskunde \& Informatica},
  city={Amstedam},
  country={The Netherlands}}
\affiliation[2]{organization={University of Twente},
  city={Enschede},
  country={The Netherlands}}
\affiliation[3]{organization={Technical University Eindhoven},
  city={Eindhoven},
  country={The Netherlands}}
\affiliation[4]{organization={Leiden University},
  city={Leiden},
  country={The Netherlands}}

\begin{abstract}
  Essentially all anytime-valid methods hinge on Ville's inequality to gain
  validity across time without incurring a union bound. Ville's inequality is a
  proper generalisation of Markov's inequality. It states that a non-negative
  supermartingale will only ever reach a multiple of its initial value with small
  probability. In the classic rendering
  both the lower bound
  (of zero) and the threshold are constant in time. We generalise both to
  monotonic curves. That is, we bound the probability that a supermartingale
  which remains above a given decreasing curve exceeds a given increasing
  threshold curve. We show our bound is tight by exhibiting a supermartingale
  for which the bound is an equality. Using our generalisation, we derive a clean finite-time version of the law of the iterated logarithm.
\end{abstract}


\end{frontmatter}

\section{Introduction}
We revisit Ville's inequality. To set the stage, consider a filtered probability space $(\Omega, \mathcal F, (\mathcal F_n)_{n \ge 0}, \pr)$. Recall that an adapted process $(M_n)_{n \ge 0}$ is a \emph{supermartingale} if for all time steps $n \ge 0$
\begin{equation}\label{eq:supm}
  \ex\sbrc*{M_{n+1}}{\mathcal F_n} ~\le~ M_n
  \qquad
  \text{and}
  \qquad
  \ex[\abs{M_n}] ~<~ \infty
  .
\end{equation}
The classic inequality by \citet{Ville:1939} states that a supermartingale which is bounded below will only ever reach a large threshold with small probability. In the customary presentation with zero as the lower bound, it reads

\begin{theorem}[Ville's Inequality]\label{thm:ville}
  Let $(M_n)_{n \ge 0}$ be a non-negative supermartingale. Then for every threshold $C > 0$,
  $
    \pr \set*{
      \exists n \ge 0 : M_n \ge C
    }
    ~\le~
    \frac{\ex \sbr*{M_0}}{C}
    .
  $
\end{theorem}
This strengthens Markov's inequality by including a search over time steps,
while keeping the threshold $C$ and probability
$\frac{1}{C}$ inversely proportional. This reveals that, as test statistics, supermartingales are immune
to the multiple testing problem that would otherwise require a Bonferroni-type
union bound. The statistical use of Ville's inequality stems from the fact that many events of interest can be encoded as a fixed supermartingale not growing. As such Ville's inequality is a workhorse in the fields of sequential analysis \citep{Wald52}, game-theoretic probability \citep{ShaferV19}, e-values \citep{ramdas2023savi} and many more.

\paragraph{Generalisations}
Multiple extensions of Ville's inequality have been proposed in discrete and continuous time. \cite{ruf2022composite} study a composite
null hypothesis, and they characterise the events that can be expressed as a
uniform (in $\pr$) supermartingale getting big. \cite{wang2023extended} study
the non-integrable case, dropping the right hand side condition from
\eqref{eq:supm}, and obtain a correction for the threshold/probability
relationship. \cite{ramdas2023randomized} develop a randomised sharpening of
Markov's inequality, which also extends to a Ville-type inequality under external stopping.

\paragraph{Our Question}
We investigate an orthogonal direction: what if the lower bound and threshold are not constant in time? That is, suppose supermartingale $(M_n)_{n \ge 0}$ is bounded below by $M_n \ge -f(n)$ for some increasing function $f$, and fix some increasing target threshold $g(n)$. What can we say about $\pr \set*{\exists n \ge 0 : M_n \ge g(n)}$? And what applications could this tackle? %

\paragraph{Ville \& LIL}
We focus on the paradigmatic application of Ville's inequality for deriving anytime-valid concentration inequalities of iterated-logarithm type \citep[for the history, see][Section 5.5]{ShaferV01}. These state that for $X_1,X_2,\ldots$ conditionally sub-Gaussian with mean $\mu$\footnote{That is, $\ex\sbrc{e^{\eta (X_{n+1}-\mu)}}{\mathcal F_n} \le e^{\frac{1}{2} \eta^2}$ for all $\eta \in \reals$ and $n \ge 0$.}, with empirical mean abbreviated by $\hat \mu_n = \frac{1}{n} \sum_{i=1}^n X_i$, for any confidence $\delta \in (0,1]$, we have
\begin{equation}\label{eq:subG.lil}
  \pr \set*{
    \exists n : \hat \mu_n - \mu \ge  (1+o(1)) \sqrt{2\frac{\ln \frac{1}{\delta} + \ln \ln n}{n}}
  }
  ~\le~
  \delta
  .
\end{equation}
\cite{Darling/Robbins:1967a,robbins1970statistical,DeLaPena1999ARatios,balsubramani2015sequential,kaufmann2021mixture} prove versions of~\eqref{eq:subG.lil} by applying Ville's inequality to a mixture martingale, crafted to implement a ``slicing'' weighted union bound over near-exponentially spaced time intervals. Could a more powerful or versatile  version of
Ville's inequality result in a cleaner path to form~\eqref{eq:subG.lil} with
an explicit and interpretable $o(1)$ term?

\paragraph{Partial Proof of Concept}
Our prime example of a load-bearing supermartingale with a diverging lower bound comes from \citet[Section 3.3]{squint}.
For bounded outcomes $X_1, X_2, \ldots \in [-1,+1]$ with conditional mean at most %
$\ex\sbrc*{X_{n+1}}{\mathcal F_n} \le \mu$, they build supermartingales of the form $n \mapsto$
\begin{equation}\label{eq:improper.squintform}
  \int_0^b \frac{\prod_{i=1}^n (1 + \eta (X_i-\mu))
    -1}{\eta} \dif \eta
  \quad
  \text{or}
  \quad
  \int_0^b \frac{
    e^{\eta \sum_{i=1}^n (X_i-\mu) - \frac{n}{2} \eta^2}
    -1}{\eta} \dif \eta
  .
\end{equation}
It is well known \citep[see e.g.][]{ShaferV01} that for $\eta \ge 0$
sufficiently small, the factors $1 + \eta (X_i-\mu)$ and $e^{\eta (X_i-\mu) -
  \frac{1}{2} \eta^2}$ are non-negative and have expectation bounded by one.
Hence taking products over time-steps $i = 1, \ldots, n$, and mixing over $\eta
\in [0,b]$ with positive weights results in a supermartingale. The innovation in
\eqref{eq:improper.squintform} is to mix the \emph{increments} (over the starting value $1$) with an \emph{improper} prior $\frac{1}{\eta} \dif \eta$. The resulting processes~\eqref{eq:improper.squintform} are \emph{bona fide} supermartingales, satisfying both parts of~\eqref{eq:supm}, but they are not bounded below by any constant, but instead by $- 1 - \ln n$ at time $n$. %
If we assume that supermartingale~\eqref{eq:improper.squintform} is not high, a concentration inequality of iterated-logarithm type follows deterministically. But why would it not be high? In the context of individual sequence prediction, a super-martingale can be prevented from growing by the method of defensive forecasting. %
In that context, the key advantage of improper vs standard mixtures is computational: the integral required for defensive forecasting admits a closed form.
Instead, in the context of stochastic concentration inequalities the growth of a non-negative supermartingale is controlled by Ville's inequality. But how to control~\eqref{eq:improper.squintform} and friends? And what are the trade-offs?

\paragraph{Our contributions}
In this article we extend Ville's inequality to handle supermartingales with arbitrary lower bounds and thresholds. We arrive at a tight characterisation of the threshold-crossing probability for any monotonic lower bound and threshold curves. Comparing to \cite{squint}, we find a way to handle sub-Gaussian distributions with unbounded outcomes (for which~\eqref{eq:improper.squintform} would be unbounded below at every $n$). Using our extensions, we prove a cleaner iterated logarithm concentration inequality.

\section{Main Result}
We show the following generalisation of Ville's inequality Theorem~\ref{thm:ville}.

\begin{theorem}\label{thm:main}
  Consider two non-decreasing functions $f$ and $g$ defined on $\set{0,1,\ldots}$ such that $-f(0) < g(0)$.
  \begin{enumerate}[(a)]
  \item \label{it:ubd}
    For any supermartingale $(M_n)_{n \ge 0}$ bounded below by $M_n \ge -f(n)$ for all $n \ge 0$ and with initial expectation $\ex[M_0] \in [-f(0), g(0)]$, we have
\begin{equation}\label{eq:result}
  \pr \set[\big]{
    \exists n \ge 0 : M_n \ge g(n)
  }
  ~\le~
  1
  - \frac{g(0) -  \ex[M_0]}{g(0) + f(0)} \prod_{n=1}^\infty \frac{g(n) + f(n-1)}{g(n) + f(n)}
  .
\end{equation}
\item \label{it:tight}
  For any $m \in [-f(n), g(n)]$ there is a martingale $(M_n)_{n \ge 0}$ bounded below by $M \ge -f$ and with $\ex[M_0] = m$ such that~\eqref{eq:result} holds with equality.
\end{enumerate}
\end{theorem}
We recover Theorem~\ref{thm:ville} for constant
 $f(n) = 0$ and $g(n) = C$ in~\eqref{eq:result}. In Section~\ref{sec:tight} we discuss the
martingale that proves tightness~\eqref{it:tight}, and
then we prove the upper bound~\eqref{it:ubd} in Section~\ref{sec:ubd}. A more
interpretable bound and examples of the use of Theorem \ref{thm:main} are
obtained in Section~\ref{sec:continuous-bound}.

\subsection{Proof of Tightness (Theorem~\ref{thm:main}\eqref{it:tight}) by Floor
  Hugging}\label{sec:tight}
We will first consider the case $M_0 = -f(0)$ deterministically. Then to create a witness martingale for $\ex[M_0] \in (-f(0), g(n)]$, we simply randomise (with weights chosen to control the mean) between starting at either $M_0 = -f(0)$ or $M_0 = g(0)$.

We now define the \emph{floor-hugger} martingale. It starts with $M_0 = -f(0)$.
If it finds itself at time $n$ at value $M_n = -f(n)$, then it either jumps up
to $M_{n+1} = g(n+1)$ or drops down to $M_{n+1} = -f(n+1)$---it hugs the floor. The martingale property forces the jump probability $p_n$ to satisfy
\begin{equation}\label{p.jump}
  -f(n) = p_n g(n+1) - (1-p_n) f(n+1)
  \quad
  \text{i.e.}
  \quad
  p_n
  =
  \frac{f(n+1)-f(n)}{g(n+1) + f(n+1)}.
\end{equation}
Let $s(n) \df \pr \delc{
    \forall t \ge n : M_t = -f(t)
  }{M_n = -f(n)}$ denote the the floor-hugger probability to never hit $M = g$, starting from time $n$ at value $-f(n)$. Then
\begin{align}\label{eq:h}
  s(n)
  &~=~
  \prod_{t=n}^\infty \del*{
    1 - \frac{f(t+1)-f(t)}{g(t+1) + f(t+1)}
  }.
\end{align}
We observe that $s(n) \in [0,1]$. Taking $n=0$, we see that the floor-hugger
martingale witnesses~\eqref{eq:result} with equality.

\subsection{Proof of Upper Bound (Theorem~\ref{thm:main}\eqref{it:ubd}) using Classic Ville}\label{sec:ubd}
Fix any supermartingale $(M_n)_{n \ge 0}$ bounded below by $M_n \ge -f(n)$. Based on it, we construct the auxiliary non-negative supermartingale $(K_n)_{n \ge 0}$ by
\begin{equation}\label{eq:K}
  K_n
  ~\df~
  1 - \frac{g(n) - M_n}{g(n) + f(n)} s(n)
  ,
\end{equation}
where $s(n)$ is the probability, defined in~\eqref{eq:h}, that the floor-hugger martingale never reaches $g$ when starting from value $-f(n)$ at time $n$.
We prove that $(K_n)_{n \ge 0}$ is a non-negative supermartingale, and then
derive our upper bound using Ville's inequalty.

\begin{lemma}
  The process $(K_n)_{n \ge 0}$ in~\eqref{eq:K} is a non-negative supermartingale.
\end{lemma}

\begin{proof}
First, $K_n$ is non-decreasing in $M_n \ge -f(n)$, and hence non-negativity follows by
$
  K_n
  \ge
  1-s(n)
  \ge
  0
  $.
  For the supermartingale claim, we have
\begin{align*}
  \ex\sbrc*{K_{n+1}}{\mathcal F_n}
  &~\le~
  1 - \frac{g(n+1) - M_n}{g(n+1) + f(n+1)} s(n+1)
  ~=~
  1 - \frac{g(n+1) - M_n}{g(n+1)+f(n)} s(n)
  \\
  &~\le~
  1 - \frac{g(n) - M_n}{g(n)+f(n)} s(n)
  ~=~
  K_n
  ,
\end{align*}
where we used the definition of $K_n$, the fact that $M_n$ is a supermartingale,
the recurrence
$
  s(n)
  =
  \del*{
    1 - \frac{f(n+1)-f(n)}{g(n+1) + f(n+1)}
  }
  s(n+1)
  $ by~\eqref{eq:h}, and $M_n \ge -f(n)$.
\end{proof}

\begin{corollary}
  The first claim of Theorem~\ref{thm:main}, i.e.\ \eqref{eq:result} holds.
\end{corollary}

\begin{proof}
As $s(n) \in [0,1]$, $K_n$ is non-decreasing in $M_n$ and hence $M_n \ge g(n)$ implies $K_n \ge 1$. Hence by Ville's inequality (Theorem~\ref{thm:ville}), we find
\begin{align*}
  \pr \set*{
    \exists n \ge 0 : M_n \ge g(n)
  }
  &~\le~
  \pr \set*{
    \exists n \ge 0 : K_n \ge 1
  }
  \\
  &~\le~
  \ex[K_0]
  ~=~
  1
  - \frac{g(0) - \ex[M_0]}{g(0) + f(0)} s(0)
  ,
\end{align*}
which, upon rewriting, yields~\eqref{eq:result}.
\end{proof}

The two key properties used in the proof are that the auxiliary martingale $(K_n)_{n\geq 0}$ is non-negative and that $M_n\geq g(n) \Rightarrow K_n\geq 1$.
However, the definition of $K_n$ might seem unintuitive because it involves $s(n)$, which consists of future values of $f$ and $g$.
Could one design a non-trivial non-negative martingale $(L_n)_{n \ge 0}$ based on the history, $L_n=L_n((M_i)_{i\leq n}, (g(i))_{i\leq n},(f(i))_{i\leq n})$, such that $M_n\geq g(n)\Rightarrow L_n \geq 1$?
For instance, one could consider defining $L_n$ as a scaled and/or translated version of $M_n$ based on past and present values of $f$ and $g$ (e.g.\ \eqref{eq:K} without the factor $s(n)$).
Such a construction turns out to be impossible.
This follows from the fact that we must have $\ex\sbrc*{L_{n+1}}{\mathcal F_n}\leq L_n$ regardless of what the original martingale $(M_n)_{n\geq 0}$ does after time $n$; in particular, this must hold if it does a ``floor-hugger'' type jump (see Section~\ref{sec:tight}) from $M_n$ to $g(n+1)$.
If this jump is succesful, then $M_{n+1}=g(n+1)$, which would imply that $L_{n+1}\geq 1$.
However, the probability of success~\eqref{p.jump} can be made arbitrarily close to one by taking $f(n+1) \to \infty$, which gives $\ex\sbrc*{L_{n+1}}{\mathcal F_n}=1\leq L_n$.
It follows that $L_n\geq 1$ for all $n$, so the implication $M_n\geq g(n) \Rightarrow L_n\geq 1$ is powerless.
The only way to avoid this is for $K_n$ to account for future values of $f$ and $g$.

\section{A Continuous Bound and Examples}\label{sec:continuous-bound}
We present a more interpretable relaxation of~\eqref{eq:result} by considering lower bound $f$ and threshold $g$ defined on real-valued inputs.
\begin{corollary}\label{cor:cont}
  Let $f,g : \reals_+ \to \reals_+$ be increasing with  $f$ differentiable. Under the setup of
  Theorem~\ref{thm:main}\eqref{it:ubd},
\[
  \pr\set*{\exists n \ge 0 : M_n \ge g(n)}
  ~\le~
  1 - \frac{g(0) -  \ex[M_0]}{g(0) + f(0)} \exp \del*{- \int_{0}^\infty \frac{f'(t)}{g(t) + f(t)}  \dif t}
  .
\]
\end{corollary}
\begin{proof}
  For every $n \ge 1$, as $g$ is increasing,
\[
  - \int_{n-1}^n \frac{f'(t)}{g(t) + f(t)} \dif t
  ~\le~
  - \int_{n-1}^n \frac{f'(t)}{g(n) + f(t)} \dif t
  ~=~
  \ln \del*{
    1 -
    \frac{f(n) - f(n-1)}{g(n) + f(n)}
  }.
  \qedhere
\]
\end{proof}
This bound is convenient in discrete time, and tight in continuous time (by an analogue of the floor hugger from Section~\ref{sec:tight}). Furthermore, its right-hand side is pleasantly invariant under reparametrisations of time by any increasing bijection.
Next we explore convenient choices for $f$ and $g$.
\begin{example}[Quadratic]\label{eq:quadr.g}
  Fix increasing differentiable $f$ with $f(0) = 0$ and $\lim_{n \to \infty} f(n) = \infty$. We propose to study $g(t) = (a f(t) + b)^2-f(t)$, which is increasing when $2 a b \ge 1$. The exponent in Corollary~\ref{cor:cont} evaluates to
  \[
    {- \int_{0}^\infty \frac{f'(t)}{g(t) + f(t)}  \dif t}
    ~=~
  {- \int_{0}^\infty \frac{f'(t)}{(a f(t) + b)^2}  \dif t}
  ~=~
  {
    \left.
      \frac{1}{a(a f(t) + b)}
      \right|_0^\infty
    }
    ~=~
  {
    - \frac{1}{a b}
  }
  .
\]
Equating the threshold crossing probability bound $1-e^{- \frac{1}{a b}}$ to $\delta$ requires picking
$a b = \frac{1}{- \ln (1 - \delta)}$.
\end{example}

The quadratic is in fact a member of a more general construction, which we present in \ref{sec:expcon}. If $g$ vs $f$ grows too slowly, \eqref{eq:result} trivialises to $1$:

\begin{example}[Counterexample]
  For diverging $f(t)$, when we take $g = O(f)$, the upper
  bound in Corollary~\ref{cor:cont} is the trivial $1$: let $t_0 \ge 0$ and $c>0$ be such that $g(t) \le c f(t)$ for all $t \ge t_0$. Then
  \[
    \int_0^\infty \frac{f'(t)}{g(t) + f(t)} \dif t
    ~\ge~
    \int_{t_0}^\infty \frac{f'(t)}{(1+c) f(t)} \dif t
    ~=~
    \left. \ln f(n) \right|_{t_0}^\infty
    ~=~
    \infty
    .
  \]
\end{example}

\section{Finite-time Law of the Iterated Logarithm (LIL)}

Let $X_1, X_2, \dots$ be a sequence of independent sub-Gaussian random variables
with common conditional mean $\mu$, and let $S_n = \sum_{i \leq n} (X_i - \mu)$
be their centered sum. In this sub-Gaussian case we can derive an
iterated-logarithm bound of form~\eqref{eq:subG.lil} using the supermartingale
\begin{equation}
  \label{eq:gaussian_martingale}
  M_n
  =
  \frac{1}{\sqrt{2\pi}}
  \int_{-\infty}^{\infty}
  \bracks{
    \exp\paren{
      \eta S_n
      -
      \frac12\eta^2n
    }
    -1
  }
  \rme^{-\eta^2/2}
  \frac{\rmd\eta}{\abs{\eta}}.
\end{equation}
This choice of $M_n$ is inspired by~\eqref{eq:improper.squintform}, but now
instead of using the improper ``prior''
$\frac{\mathbf{1}\{|\eta| \leq b\}}{|\eta|}\rmd \eta$,
\eqref{eq:gaussian_martingale} uses
$\frac{\rme^{-\eta^2 / 2}}{|\eta|}\rmd \eta$, which eases some Gaussian-integral
computations. Lemma~\ref{lem:gaussian_rewrite} shows that
$M_n \geq - \frac{1}{\sqrt{2\pi}}\ln(1 + n)$. Thus, we can apply
Theorem~\ref{thm:main} with $f(n) = \frac{1}{\sqrt{2\pi}}\ln(1 + n)$. The
following is shown using a specific choice of threshold $g$ in
Theorem~\ref{thm:main} and a linearization of the ensuing implicit bound. A slightly tighter version is in \ref{sec:proof_simpler_lil}.

\begin{lemma}\label{lem:simpler_lil} Let $X_1, X_2, \dots$ be a sequence of
  independent sub-Gaussian random variables
  with common conditional mean $\mu$, and let
  $S_n = \sum_{i \leq n} (X_i - \mu)$. As long as $\delta\leq 3/5$, with
  probability larger than $1 - \delta$ and $\delta' \df -\ln(1 - \delta)$,
  \begin{equation*}
    \forall n :
    \abs{
    \frac{S_n}{\sqrt{n + 1}}
    }
    \leq
    \sqrt{2\ln\bracks{1 +
        \frac{\ln(n + 1) /\sqrt{\pi} + \rme\sqrt{2}}{\delta'} \
    \ln^2\paren{
      \tfrac{\ln(n + 1)}{\sqrt{2\pi}} + \rme
    }
    }
    }
    +
    \frac{1}{\sqrt{2}}.
  \end{equation*}
\end{lemma}

\section{Discussion}
We extended Ville's inequality to supermartingales that are bounded below by a time-decreasing lower bound, and to time-increasing thresholds. We then put the result to use in deriving concentration inequalities of iterated-logarithm type. Let us discuss what we learned on the way.

\paragraph{Is our extended Ville's inequality strictly more general?} Here our study reveals that the answer is, surprisingly, \emph{no}. This becomes apparent in the proof of the upper bound in Section~\ref{sec:ubd}, which operates by applying Ville's inequality to an auxiliary non-negative supermartingale. Indeed if one is able to encode an event of interest directly in the form of that auxiliary supermartingale, classic Ville suffices.

\paragraph{Is our extension more user friendly?} We strongly believe that the auxiliary supermartingale from Section~\ref{sec:ubd} is not intuitive or natural, and we are not aware of any result that can be rendered as having guessed it. We certainly discovered it last, when trying to prove that the Floor Hugger martingale from Section~\ref{sec:tight} is the worst case. This suggests that possible consequences of our extension, even though accessible in principle, were not practically available. In that sense we claim our extension is more empowering.

\paragraph{Is our finite-time LIL better?}
To put our result to use, we develop a sub-Gaussian concentration inequality of iterated logarithm type. If we juxtapose the consequence of applying Ville to proper and improper mixture approaches, and ignore constants, we need to unpack respectively
\[
  \int_{-b}^b \frac{e^{\eta S_n - \frac{1}{2} n \eta^2}}{\abs{\eta} (\ln \abs{\eta})^{1+c}} \dif \eta ~\le~ \frac{1}{\delta}
  \quad
  \text{or}
  \quad
  \int_{-b}^b \frac{e^{\eta S_n - \frac{1}{2} n \eta^2} - 1}{\abs{\eta}} \dif \eta ~\le~ \frac{\ln n}{\delta} \del*{\ln \ln n}^{1+c}
  .
\]
Any lower bound of either integral gives a LIL bound. The comparison then boils down to which form is easier to bound tightly. We show in Lemma~\ref{lem:gaussian_rewrite} that the right-hand inequality rewrites to a fixed function of the deviation $\frac{S_n}{\sqrt{n+1}}$ being below the threshold  $\frac{\ln n}{\delta} (\ln \ln n)^{1+c}$. Due to this separation, obtaining a LIL bound reduces to inverting that function. A similar modularity is not present in the left-hand inequality, and its tight analysis is considerably more involved \citep{robbins1970statistical,squint}.

\paragraph{One-sided LIL bounds}
The proper mixture technique easily adapts to delivering one-sided LIL bounds, simply by mixing only over $\eta \ge 0$. A one-sided improper mixture analogue of~\eqref{eq:gaussian_martingale} would be well-defined. But it would not be bounded below, as can be seen by taking $S_n \to -\infty$, which takes the exponential to zero. This is a curious downside that invites further investigation.

\paragraph{Multiple Testing / Betting with Subsidies}
Theorem~\ref{thm:main} reveals that the worst-case nonnegative martingale only makes a single all-or-nothing attempt to cross the threshold. This is how Ville generalises Markov's inequality without any overhead. The story changes fundamentally with a receding lower bound $-f$. Now, every time margin is created by growing $f$, the worst-case martingale makes a new independent attempt to hit the then-current threshold $g$. These attempts all contribute to the overall probability of ever reaching $g$.
In that light, non-constant lower bounds correspond to multiple testing scenarios.
Our relation between $f,g$ and the probability of ever reaching $g$ quantifies the exact multiple testing correction required.

Yet another equivalent perspective is to consider the non-negative semimartingale $M_n + f(n)$ consisting of a stochastic supermartingale part $M_n$ and a deterministic increasing ``subsidy'' part $f(n)$. In betting language, the subsidies allow for engaging in multiple all-or-nothing bets, and consequently reduce the surprise upon sizeable gain. Our main result quantifies exactly the relation between surprise and gain $g$ as a function of the subsidy $f$.

\DeclareRobustCommand{\VANDER}[3]{#3}

\bibliographystyle{abbrvnat}
\bibliography{bib}

\DeclareRobustCommand{\VANDER}[3]{#2}

\appendix

\section{Exp-concave Dampener Example}\label{sec:expcon}
In this section we present an extension of Example~\ref{eq:quadr.g} that allows for $g(n)$ to grow slower than quadratic in $f(n)$ (but still super-linear).

\begin{example}[Exp-concave]\label{ex:exp_concave}
  Fix increasing differentiable $f$ with $f(0) = 0$ and
  $\lim_{n \to \infty} f(t) = \infty$, and let $h$ be $1$-exp-concave---the
  function $\xi\mapsto \rme^{-h(\xi)}$ is concave---and twice differentiable such
  that $\lim_{\xi \to \infty} h(\xi) = 0$ (the quadratic above corresponds to
  $h(\xi) = \frac{1}{a (a \xi+b)}$). We propose to investigate
  $
    g(t) ~=~ \frac{- 1}{h'(f(t))} - f(t)
    .
  $
  Now $g(t)$ being non-decreasing translates to
  $
    0 ~\le~
    g'(t) ~=~
    \del*{
      \frac{h''(f(t)) }{h'(f(t))^2} - 1
    } f'(t)
    ,
  $
  and given that $f'(t) \ge 0$, it remains to check $
  h'(f(t))^2
  \le
  h''(f(t))$,
which holds by the exp-concavity assumption.
Then, as $h(f(\infty)) = h(\infty) = 0$,
\[
  {- \int_{0}^\infty \frac{f'(t)}{g(t) + f(t)}  \dif t}
  ~=~
  {\int_{0}^\infty h'(f(t)) f'(t) \dif t}
   ~=~
  {\left.  h(f(t)) \right|_{0}^\infty}
  ~=~
  - h(0)
  .
\]
So to ascertain threshold crossing probability bounded by $\delta$, it suffices
to pick $1 - e^{-h(0)} = \delta$, i.e.\ $h(0) = -\ln(1 - \delta)$. Interesting
choices for $h$ include the polynomial $h(\xi) = \frac{(a \xi+b)^{1-c}}{a (c-1)}$
for $c>1$ (the quadratic being $c=2$) and $a,b > 0$ which is exp-concave
provided $a c \ge b^{1-c}$. They also include the fat-tailed
$h(\xi) = \frac{1}{\ln(a \xi + b)^{c-1}}$ for $a>0$, $b \ge e$, $c>1$, which is
exp-concave. The latter ensures $g(n) \sim f(n) (\ln f(n))^c$ as $n\to \infty$.
\end{example}

\section{Proof of~Lemma~\ref{lem:simpler_lil}}
\label{sec:proof_simpler_lil}

In this section, we show Lemma~\ref{lem:explicit_lil}, which directly
implies Lemma~\ref{lem:simpler_lil}.
\begin{lemma}\label{lem:explicit_lil} Under the set up of Lemma~\ref{lem:implicit_lil}, as long as
  $\delta\leq 3/5$, with probability larger than $1 - \delta$,
  \begin{equation*}
    \forall n :
    \abs{
    \frac{S_n}{\sqrt{n + 1}}
    }
    \leq
    \sqrt{2\ln\bracks{1 +
        \frac{\ln(n + 1) /\sqrt{\pi} + \rme\sqrt{2}}{\delta'} \
    \ln^2\paren{
      \frac{\ln(n + 1)}{\sqrt{2\pi}} + \rme
    }
    }
    }
    +
    R(n),
  \end{equation*}
  where $\delta' = -\ln(1 - \delta)$, the function
  $\ln_{(\delta)}(n) = 1 + \frac{\ln(1 + n)}{\delta'\sqrt{2\pi}}$, and the
  residual quantity
  $R(n) = \frac{
     \tau
    \paren{
      1 -
      \sqrt{\frac{\tau(n) }{ \tau(n) + 2\sqrt{2}} \ \frac{1}{\ln(1 + \sqrt{2}\tau(n))}}
    }
  }{
    \frac{1}{2\ln(1 + \sqrt{2}\tau(n))}
    \sqbrack{
      (1 + \sqrt{2}\tau(n))\{2\ln(1 + \sqrt{2}\tau(n)) - 1\}
      +
      \ln(1 + \sqrt{2}\tau(n))
      +
      1
    }
    \sqrt{\frac{\tau(n)}{\tau + 2\sqrt{2}}}
  }$ with $\tau(n) =  \frac{\ln(n + 1) /\sqrt{2\pi} + \rme}{\delta'} \
    \ln^2\paren{
      \frac{\ln(n + 1)}{\sqrt{2\pi}} + \rme
    }$.
  Furthermore, the remainder term $R = R(n)$ is harmless: $R \leq 1 / \sqrt{2}$,
  the function $n\mapsto R(n)$ is increasing, and
  $R(n) \sim \frac{1}{\sqrt{2}}\paren{1 - \frac{1}{\sqrt{2\ln\ln n}}}$ as $n\to \infty$.
\end{lemma}

If one considers a version $M_n^\kappa$ of $M_n$ with ``prior''
$\frac{\rme^{-\eta^2 \kappa^2/ 2}}{|\eta|}\rmd \eta$ for some $\kappa > 0$, the time-uniform high-probability bound $\forall n:S_n \leq F_\delta(n)$ from $M_n$ scales to $\forall n: S_n  \leq \kappa F_\delta(n / \kappa^2)$ from $M_n^\kappa$.

\bigskip
\noindent
Throughout this section, we will make use of the following function
\begin{equation}\label{eq:I}
I(x) := \int_0^{x} \rme^{u^2 / 2}\erf\paren{\frac{u}{\sqrt{2}}} \rmd u
\end{equation}
We begin with the following lemma (its proof can be found in
Section~\ref{sec:proof-lem:gaussian_rewrite}), which presents a rewrite of $M_n$
from~\eqref{eq:gaussian_martingale}.
\begin{lemma}\label{lem:gaussian_rewrite}
  The martingale $(M_n)_n$ from~\eqref{eq:gaussian_martingale} satisfies
  \begin{equation*}
    M_n
    =
    I\paren{\frac{S_n}{\sqrt{n + 1}}}
    -\frac{1}{\sqrt{2\pi}}\ln(1 + n)
  \end{equation*}
  where $I$ from~\eqref{eq:I} is nonnegative and consequently $M_n \geq -\frac{1}{\sqrt{2\pi}}\ln(1 + n) $.
\end{lemma}
This result decomposes $M_n$ into a contribution only dependent on the scaled
and smoothed empirical mean $\frac{S_n}{\sqrt{n+1}}$ and a contribution only
dependent on the sample size $n$. With a choice of $g$, we obtain an implicit
concentration inequality in terms of the function $I$ (see the proof in Section~\ref{sec:proof-implicit_lil}).
\begin{lemma}\label{lem:implicit_lil}
  Let $X_1, X_2, \dots$ be an i.i.d. sequence of standard normal random
  variables. Then, as long as $\delta \leq 3/5$, with probability larger than
  $1 - \delta$,
  \begin{equation}\label{eq:implicit_lil}
    \forall n : I\del*{\frac{S_n}{\sqrt{n + 1}}}
    ~\leq~
    \frac{\ln(n + 1) /\sqrt{2\pi} + \rme}{\delta'} \
    \ln^2\paren{
      \frac{\ln(n + 1)}{\sqrt{2\pi}} + \rme
    },
  \end{equation}
  where $I$ is defined in~\eqref{eq:I} and
  $\delta' = -\ln(1 - \delta)$.
\end{lemma}
The right hand side of~\eqref{eq:implicit_lil} is of order $\ln(n) + o(\ln(n))$
as $n \to \infty$, and, as we will see, $I(x) \gtrsim \rme^{x^2/2} / x$ as
$x\to\infty$. Thus, \eqref{eq:implicit_lil} implies that, uniformly over time
and with probability larger than $1 - \delta$, we have that
$S_n /\sqrt{n} \leq \sqrt{2(\ln(\ln(n)/ \delta))}(1 + o(1))$ as $n\to\infty$.

The last step of the proof, is to invert the implicit bound in
Lemma~\ref{lem:implicit_lil} to obtain an explicit bound for $S_n$. The
following lemma provides a tool to do so (its proof is in
Section~\ref{lem:invert_I}).
\begin{lemma}\label{lem:invert_I}
  Let $I$ be as in Lemma~\ref{lem:implicit_lil} and let $\tau > 0$. Then, if
  $I(x) \leq \tau$, then
  \begin{equation*}
    |x| \leq
    \sqrt{2\ln(1 + \sqrt{2}\tau)}
    +
    R(\tau)
  \end{equation*}
  with $R(\tau) = \frac{
     \tau
    \paren{
      1 -
      \sqrt{\frac{\tau }{ \tau + 2\sqrt{2}} \ \frac{1}{\ln(1 + \sqrt{2}\tau)}}
    }
  }{
    \frac{1}{2\ln(1 + \sqrt{2}\tau)}
    \sqbrack{
      (1 + \sqrt{2}\tau)\{2\ln(1 + \sqrt{2}\tau) - 1\}
      +
      \ln(1 + \sqrt{2}\tau)
      +
      1
    }
    \sqrt{\frac{\tau}{\tau + 2\sqrt{2}}}
  }$, which satisfies the following:   $R(\tau) \leq 1/\sqrt{2}$, and  $\tau
  \mapsto R(\tau)$ is increasing, and $R(\tau) \sim \frac{1}{\sqrt{2}}\paren{1
  - \frac{1}{\sqrt{\ln\tau}}}$ as $\tau\to\infty$.
\end{lemma}

With these results at hand, we can proof Lemma~\ref{lem:explicit_lil}.
\begin{proof}[Proof of Lemma~\ref{lem:explicit_lil}]
  Using the Lemma~\ref{lem:implicit_lil}, we obtain a uniform bound for
  $I(S_n / \sqrt{n + 1})$. Using Lemma~\ref{lem:invert_I}, we obtain the
  explicit bound for $S_n$ that is claimed.
\end{proof}

The next sections are respectively dedicated to the proofs of
Lemma~\ref{lem:gaussian_rewrite}, Lemma~\ref{lem:implicit_lil}, , and
Lemma~\ref{lem:invert_I}.

\subsection{Proof of Lemma~\ref{lem:gaussian_rewrite}}
\label{sec:proof-lem:gaussian_rewrite}
\begin{proof}[Proof of Lemma~\ref{lem:gaussian_rewrite}]
We use Feynman's trick.
It begins by introducing an extraneous factor into the integral that defines $M_n$.
Define
\begin{equation*}
  v(x)
  :=
  \frac{1}{\sqrt{2\pi}}
  \int_{-\infty}^\infty\bracks{
    \exp
    \paren{
      \eta  S_n
      -
      \frac{1}{2}\eta^2 n
    }
    - 1
  }
  \rme^{-\eta^2 x^2 / 2}
  \frac{\rmd \eta}{|\eta|}
\end{equation*}
and obtain
\begin{equation}\label{eq:vderiv}
  v'(x)
  ~=~
  \frac{\sqrt{\frac{2}{\pi }} n}{n x+x^3}-\frac{S_n x \rme^{\frac{S_n^2}{2 \left(n+x^2\right)}} \erf\left(\frac{S_n}{\sqrt{2(n+x^2)}}\right)}{\left(n+x^2\right)^{3/2}}.
\end{equation}
Now we compute $M_n=v(1) = - \int_1^\infty v'(x) \dif x$. The first term on the
right hand side of~\eqref{eq:vderiv} contributes
$-\frac{\ln (n+1)}{\sqrt{2 \pi }}$ to the integral. For the second term, we
need to compute
\[
  \int_1^\infty
  \frac{S_n x \rme^{\frac{S_n^2}{2 \left(n+x^2\right)}} \erf\left(\frac{S_n}{\sqrt{2(n+x^2)}}\right)}{\left(n+x^2\right)^{3/2}}
  \dif x.
\]
After reparameterising by
$u(x) = \frac{S_n}{\sqrt{n + x^2}}$ with
$u'(x) = - \frac{S_n x}{(x^2 + n)^{3 / 2}}$, this can be written as
\[
  \int_{u(1)}^{u(\infty)}
  -\rme^{\frac{u^2}{2}} \erf\left(\frac{u}{\sqrt{2}}\right)
  \dif u
  ~=~
  \int_{0}^{\frac{S_n}{\sqrt{n + 1}}}
  \rme^{\frac{u^2}{2}} \erf\left(\frac{u}{\sqrt{2}}\right)
  \dif u.
\]
All in all, this gives
\[
  M_n
  ~=~
  - \frac{\ln (n+1)}{\sqrt{2 \pi }}
  + \int_{0}^{\frac{S_n}{\sqrt{n + 1}}}
  \rme^{\frac{u^2}{2}} \erf\left(\frac{u}{\sqrt{2}}\right)
  \dif u
\]
\end{proof}

\subsection{Proof of Lemma~\ref{lem:implicit_lil}}
\label{sec:proof-implicit_lil}

Lemma~\ref{lem:gaussian_rewrite} allows us to use Theorem~\ref{thm:main} with
lower bound $f(x) = \ln(1 + x) / \sqrt{2\pi}$. All that is left is to choose a threshold function $g(x)$. The next lemma shows such a choice (its proof is
in~\ref{sec:proof-lem:g_is_good}).
\begin{lemma}\label{lem:g_is_good}
  Let $f(x) = \frac{1}{\sqrt{2\pi}} \ln(1 + x)$, let $u$ be the function
  $u(x) = \frac{1}{d}\paren{\rme + x} \ln\paren{\rme + x}^2, $ and let
  $g(x) := u(f(x)) - f(x)$. Then, the following two hold.
  \begin{enumerate}
  \item If $0< d \leq 1$, then $g \geq f$, and the function $g$ is increasing.
  \item $\int_0^\infty \frac{f'(x)}{f(x) + g(x)}\rmd x = d$
  \end{enumerate}
\end{lemma}

We can now proceed to prove Lemma~\ref{lem:implicit_lil}.
\begin{proof}[Proof of Lemma~\ref{lem:implicit_lil}]
  By Lemma~\ref{lem:gaussian_rewrite}, the martingale $M_n$ can be rewritten as
  $M_n = I(S_n / \sqrt{n + 1}) - \ln(1 + n)/\sqrt{2\pi}$ with
  $I$ defined in~\eqref{eq:I}. Use
  Theorem~\ref{thm:main} with $f(x) = \ln(1 + x)/ \sqrt{2\pi}$ and
  $g(x) = u(f(x)) - f(x)$ with
  $u(x) = \frac{1}{\delta'} \paren{  e+x }  \ln\paren{e +
     x }^2$ where $\delta' = -\ln(1 - \delta)$. By
  Lemma~\ref{lem:g_is_good} (with $d = \delta'$) this is a valid choice of $g$
  as long as $-\ln(1 - \delta) \leq 1$, that is, as long as
  $\delta \leq 1 - \rme^{-1}\approx 0.63$. Theorem~\ref{thm:main} then implies
  that, with probability smaller than
  $1 - \exp\paren{-\int_0^\infty\frac{f'(u)}{f(u) + g(u)}}\rmd u = 1 -
  \exp(\delta') = \delta$,
  \begin{equation*}
    \exists n:
    I(S_n / \sqrt{n + 1})
    \geq
    u(f(n)).
  \end{equation*}
  This is equivalent to the claim that we set ourselves to prove.
\end{proof}

\subsubsection{Proof of Lemma~\ref{lem:g_is_good}}
\label{sec:proof-lem:g_is_good}

\begin{proof}[Proof of Lemma~\ref{lem:g_is_good}]
  The function $u$ is convex, $u(0) = \rme / d$, and $u'(0) = 3/c$. Then, the
  function $u$ is bounded by its tangent at zero, that is
  $u(x) \geq \rme /d 3x / d$, which is larger than $\rme + 3x$ because
  $0 < d\leq 1$. Then $g(x) = u(f(x)) - f(x) \geq 2f(x) + \rme \geq f(x)$, which
  proves the first claim. The second claim---that $g$ is increasing---follows
  because $x\mapsto f(x)$ is increasing and $x\mapsto u(x) - x$ is convex, and,
  as shown before, $u'(0) - 1 \geq 0$. The third claim follows because the $u$
  was chosen so that $u = -\frac{1}{h'}$ with $h(x) = \frac{d}{\ln(\rme + x)}$
  (see Example~\ref{ex:exp_concave}), and
  $\int_0^\infty \frac{f'(x)}{f(x) + g(x)}\rmd x =
  \int_{0}^{\infty}\frac{f'(x)}{u(f(x))}\rmd x = -\int_0^\infty h'(x)\rmd x =
  h(0) -h(\infty) = d$.
\end{proof}

\subsection{Proof of Lemma~\ref{lem:invert_I}}
\label{sec:proof_invert_I}

We would like, for a given threshold $\tau$, to find a $\xi$ so that $x\leq \xi$
anytime that $I(x) \leq \tau$. Ideally, $\xi = I^{-1}(\tau)$, but we do not have
access to an explicit form for $I^{-1}$. We begin by given a lower bound on $I$
that is more amenable to explicit computation.
\begin{lemma}\label{lem:lower_bound_I}
  With $I$ defined in~\eqref{eq:I},
  $I(x)\geq \ell(x) := \sqrt{ \frac{ (\rme^{x^2 / 2} - 1)^{3} }{ x^2\rme^{x^2 /
        2} } }$.
\end{lemma}
The function $\ell$ from Lemma~\ref{lem:lower_bound_I} is convex, and any tangent is a
lower bound. Note additionally that
$\ell(x)\sim (\rme^{x^2 / 2} - 1) / \sqrt{2}$ as $x\downarrow 0$ so that
$\ell^{-1}(\tau) \sim \sqrt{2\ln(1 + \sqrt{2}\tau)}$ as $\tau\downarrow 0$. This
is behind the next lemma (its proof is in Section~\ref{sec:proof_invert_ell}).
\begin{lemma}
  \label{lem:invert_ell}
  Assume that $\tau > 0$ and $x$ is such that $\ell(x) \leq \tau$, where $\ell$
  is as in Lemma~\ref{lem:lower_bound_I}. Then, for any $\xi > 0$
  \begin{equation*}
    |x| \leq \xi + \frac{\tau - \ell(\xi)}{\ell'(\xi)}.
  \end{equation*}
  In particular, if $\xi = \sqrt{2\ln(1 + \sqrt{2}\tau)}$, then
  \begin{equation*}
    |x| \leq
    \sqrt{2\ln(1 + \sqrt{2}\tau)}
    +
    R(\tau)
  \end{equation*}
  with $R(\tau) = \frac{
     \tau
    \paren{
      1 -
      \sqrt{\frac{\tau }{ \tau + 2\sqrt{2}} \ \frac{1}{\ln(1 + \sqrt{2}\tau)}}
    }
  }{
    \frac{1}{2\ln(1 + \sqrt{2}\tau)}
    \sqbrack{
      (1 + \sqrt{2}\tau)\{2\ln(1 + \sqrt{2}\tau) - 1\}
      +
      \ln(1 + \sqrt{2}\tau)
      +
      1
    }
    \sqrt{\frac{\tau}{\tau + 2\sqrt{2}}}
  }$, which satisfies the following:   $R(\tau) \leq 1/\sqrt{2}$, and  $\tau
  \mapsto R(\tau)$ is increasing, and $R(\tau) \sim \frac{1}{\sqrt{2}}\paren{1
  - \frac{1}{\sqrt{\ln\tau}}}$ as $\tau\to\infty$.
\end{lemma}

With this, the proof of Lemma~\ref{lem:invert_I} is immediate.
\begin{proof}[Proof of Lemma~~\ref{lem:invert_I}]
  By Lemma~\ref{lem:lower_bound_I}, $I \leq \ell$. Lemma~\ref{lem:invert_ell}
  then implies the result.
\end{proof}

\subsubsection{Proof of Lemma~\ref{lem:lower_bound_I} }
\label{sec:proof_lowerboundI}

\begin{proof}[Proof of Lemma~\ref{lem:lower_bound_I}]
  Through a series of rewrites, we will show that
  \begin{equation}\label{eq:sech_rewrite}
    I(x)
    =
   \sqrt{\frac{2}{\pi}}\int_0^\infty \bracks{\rme^{\frac{1}{2}x^2 \mathrm{sech}^2(w)} - 1} \rmd w.
 \end{equation}
 Once this is proven, we use that
  \begin{equation*}
    \ln\paren{\rme^{\frac{1}{2} x^2 \mathrm{sech}(w)} - 1}
    \geq
    \ln(\rme^{x^2 /2} - 1) - \frac12 v(x)w^2
  \end{equation*}
  with $v(x) = \frac{x^2\rme^{x^2 / 2}}{(\rme^{x^2 / 2} - 1)}$ and that
  consequently
  \begin{equation*}
    I(x)
    \geq
    \sqrt{\frac{2}{\pi}}
    (\rme^{x^2 / 2} - 1)
    \int_0^\infty \rme^{-v(x)w^2} \rmd w
    =
    \sqrt{
      \frac{
        (\rme^{x^2 / 2} - 1)^{3}
      }{
        x^2\rme^{x^2 / 2}
      }
    } =
    \ell(x),
  \end{equation*}
  which is the claim. Now we show~\eqref{eq:sech_rewrite}.

  Writing out the defintion of the error function and a change of variables
  shows that
  \begin{align*}
    \int_0^{x} \rme^{u^2 / 2}\erf\paren{\frac{u}{\sqrt{2}}} \rmd u
    &=
    \frac{2}{\sqrt{\pi}}
    \int_0^x
    \int_0^{u/\sqrt{2}}\rme^{u^2 / 2 - v^2 }\rmd v \rmd u\\
    &=
    \sqrt{\frac{2}{\pi}}
    \int_0^x
    \int_0^{u}\rme^{u^2 / 2 - v^2 / 2 }\rmd v \rmd u.
  \end{align*}
  Use the change of variables $a = (u - v) / \sqrt{2}$ and
  $b = (u + v)/\sqrt{2}$. Its Jacobian is 1 and its inverse is given by
  $u = (a + b) / \sqrt{2}$ and $v = (b - a) / \sqrt{2}$. The integral becomes
  \begin{equation}
    I
    =
    \sqrt{\frac{2}{\pi}}
    \int_0^{x'}\int_a^{2x' - a}\rme^{ab}\rmd b\rmd a,
  \end{equation}
  where $x' = x/\sqrt{2}$. Use now the change of variables $u = \ln(\sqrt{b /
    a})$ and $v = \sqrt{ab}$. Then,
  \begin{align}
    I
    &=\label{eq:sech_last_step}
      \sqrt{\frac{2}{\pi}}
      \int_0^\infty\int_0^{x'\sqrt{1 - \tanh^2(u)}} \rme^{v^2}2v\rmd v\rmd u\\
    &=
      \sqrt{\frac{2}{\pi}}
      \int_0^\infty\bracks{\rme^{x'^2\mathrm{sech}^2(u)} - 1}\rmd u.
  \end{align}
  This concludes the proof.
\end{proof}

\subsubsection{Proof of Lemma~\ref{lem:invert_ell}}
\label{sec:proof_invert_ell}
\begin{proof}[Proof of Lemma~\ref{lem:invert_ell}]
  The function $\ell$ is convex and increasing. Convexity implies that any
  tangent line to $\ell$ is a lower bound on it. Hence, for any point $\xi$ and
  any $x$, we have that $\ell(\xi) + \ell'(\xi)(x - \xi) \leq \ell(x)$. This implies
  the first claim. The derivative of $\ell$ is
  \begin{equation*}
    \ell'(x)
    =
    \frac{\mathrm{sign}(x)}{x^2}
    \sqrt{\frac{\rme^{x^2 / 2} - 1}{\rme^{x^2 / 2}}}
    \paren{
      \rme^{x^2 / 2}(x^2 - 1)
      +
      \frac12 x^2
      +
      1
    }
  \end{equation*}
  Thus, using the first claim with $\xi = \sqrt{2\ln(1 + \sqrt{2}\tau)}$, we
  obtain that
  $|x| \leq \sqrt{2\ln(1 + \sqrt{2}\tau)} + \frac{\sqrt{2\ln(1 + \sqrt{2}\tau)}
    - \ell(\sqrt{2\ln(1 + \sqrt{2}\tau)})}{\ell'(\sqrt{2\ln(1 + \sqrt{2}\tau)})}$.
  Writing this last term explicitly, gives the expresion for $R$ in the claim.
  The fact that $R(\tau)$ is increasing follows from the convexity of $\ell$ and
  the asymptotic expression is readily obtained from the explicit form of $R$.
  This is all that was to be proven.
\end{proof}

\end{document}